\documentclass[a4paper,12pt]{article}

\usepackage[left=2cm,right=2cm, top=2cm,bottom=3cm,bindingoffset=0cm]{geometry}

\usepackage{verbatim}
\usepackage{amsmath}
\usepackage{amsthm}
\usepackage{amssymb}
\usepackage{delarray}
\usepackage{cite}
\usepackage{hyperref}
\usepackage{mathrsfs}
\usepackage{tikz}
\usetikzlibrary{patterns}
\usepackage{caption}
\DeclareCaptionLabelSeparator{dot}{.}
\newcommand{\al}{\alpha}
\newcommand{\be}{\beta}

\newcommand{\de}{\delta}
\newcommand{\la}{\lambda}
\newcommand{\om}{\omega}

\newcommand{\eps}{\varepsilon}

\theoremstyle{plain}

\numberwithin{equation}{section}

\newtheorem{thm}{Theorem}[section]
\newtheorem{lem}[thm]{Lemma}
\newtheorem{prop}[thm]{Proposition}
\newtheorem{cor}[thm]{Corollary}

\theoremstyle{definition}

\newtheorem{ip}[thm]{Inverse Problem}

\theoremstyle{remark}

\newtheorem{remark}[thm]{Remark}

\DeclareMathOperator*{\Res}{Res}

 \allowdisplaybreaks

\begin{document}

\begin{center}
{\Large\bf Third-order inverse spectral problem \\[0.2cm]
with the three-point boundary conditions}
\\[0.5cm]
{\bf Natalia P. Bondarenko}$^{1,2,3}$\\[0.1cm]

\textit{email:} bondarenkonp@sgu.ru\\[0.2cm]

1. Department of Mechanics and Mathematics, Saratov State University, 
Astrakhanskaya 83, Saratov 410012, Russia, \\

2. Department of Applied Mathematics, Samara National Research University, Moskovskoye Shosse 34, Samara 443086, Russia, \\

3. S.M. Nikolskii Mathematical Institute, RUDN University, 6 Miklukho-Maklaya St, Moscow, 117198, Russia.
\end{center}

\vspace{0.5cm}

{\bf Abstract.} In this paper, we study a new inverse spectral problem that consists in the recovery of the third-order differential equation from two spectra corresponding to the boundary conditions 
$y(0) = y(1) = y(2) = 0$ and $y(0) = y'(0) = y(1) = 0$. The uniqueness and existence theorems for the solution are obtained. To prove the results, we treat the inverse problem using a general approach that reconstructs higher-order differential operators from the Weyl-Yurko matrix.

\medskip


\medskip


\vspace{1cm}

\section{Introduction} \label{sec:intr}

Consider the third-order differential equation
\begin{equation} \label{eqv}
y''' + (p y)' + p y' + \mathrm{i} q y = \la y, \quad x \in (0, 2),
\end{equation}
where $p \in L_1(0, 1)$ and $q \in W_1^{-1}(0, 1)$ are real-valued $1$-periodic functions, $\la$ is the spectral parameter. Equation \eqref{eqv} with distribution coefficient $q$ is understood in the sense of Mirzoev-Shkalikov regularization \cite{MS16, MS16-talk, MS19} (see Section~\ref{sec:prelim} for details).

Denote by $\{ \la_n \}_{n \ge 1}$ the eigenvalues of the boundary value problem $\mathcal L = \mathcal L(p, q)$ for equation \eqref{eqv} on the interval $(0, 1)$ with the boundary conditions
\begin{equation} \label{bc1}
y(0) = y'(0) = 0, \quad y(1) = 0,
\end{equation}
and by $\{ \mu_n \}_{n \in \mathbb Z_0}$ ($\mathbb Z_0 := \mathbb Z \setminus \{ 0 \}$) the eigenvalues of the boundary value problem $\mathcal M = \mathcal M(p, q)$ for equation \eqref{eqv} with the three-point Dirichlet boundary conditions
\begin{equation} \label{bc2}
y(0) = y(1) = y(2) = 0.
\end{equation}

In this paper, we obtain the uniqueness theorem and solvability conditions for the following inverse spectral problem:

\begin{ip} \label{ip:1}
Given the spectra $\{ \la_n \}_{n \ge 1}$ and $\{ \mu_n \}_{n \in \mathbb Z_0}$, find $p$ and $q$.
\end{ip}

Inverse problems of spectral analysis have been most comprehensively studied for second-order differential operators (see the monographs \cite{Mar77, Lev84, FY01, Krav20} and references therein). Differential operators of higher orders are significantly more difficult to investigate, because the famous method of Gelfand and Levitan \cite{GL51}, based on transformation operators, does not work for them. 

The most effective approach to higher-order inverse problems has been developed by Yurko \cite{Yur92, Yur00, Yur02}. Specifically, he was the first to introduce the spectral characteristic (the so-called Weyl-Yurko matrix) that uniquely determines the coefficients of a linear differential equation of any integer order $n \ge 2$ for any behavior of the spectrum. Furthermore, Yurko reduced the corresponding nonlinear inverse problems to linear equations in suitable Banach spaces by using the idea of contour integration of spectral mappings from \cite{Leib71}.
As a result, procedures for constructive solutions were developed and necessary and sufficient conditions of solvability were obtained for higher-order inverse problems with integrable coefficients on a finite interval and on the half-line (see \cite{Yur02}). Recently, the method of Yurko was generalized to differential operators with distribution coefficients \cite{Bond21, Bond22-alg, Bond24, GYB25}. Among the other approaches to higher orders, we mention the studies \cite{Sakh61, Khach83, Mal94}, based on transformation operators under specific restrictions on differential equation coefficients, and the approach of Beals et al \cite{Beals88} to inverse scattering on the line.

The third-order inverse problems have applications to integration of the nonlinear Boussinesq equation \cite{McK81, DTT82, Jur91, BK24}. Reconstruction of the third-order differential operators from various types of spectral and scattering data was investigated in \cite{Bond23, ATU25, Zol25, BK26} and other studies. In particular, Badanin and Korotyaev in the recent paper \cite{BK26} solved the inverse spectral problem with the three-point Dirichlet conditions \eqref{bc2} by using the eigenvalues $\{ \mu_n \}_{n \in \mathbb Z_0}$ and the corresponding norming constants. However, their problem statement and the results are valid only locally, for sufficiently small norms of the coefficients $p$ and $q$.

In this paper, we introduce the new inverse problem statement (Inverse Problem~\ref{ip:1}) for the third-order differential equation \eqref{eqv}. Namely, we consider the three-point spectrum $\{ \mu_n \}_{n \in \mathbb Z_0}$ together with such data (the second spectrum $\{ \la_n \}_{n \ge 1}$) that guaranty the global uniqueness of the inverse problem solution. Furthermore, we find sufficient conditions, close to the necessary ones, for the existence of a solution.

Let us formulate the main results. Along with~\eqref{eqv}, we consider another equation of the same form but with different coefficients $\tilde p$ and $\tilde q$. We agree that, if a symbol $\al$ denotes an object related to $(p, q)$, then the symbol $\tilde \al$ with tilde will denote the analogous object related to $(\tilde p, \tilde q)$.

\begin{thm} \label{thm:uniq}
Suppose that $\la_n \ne -\overline{\la_k}$ for $n, k \ge 1$. Then the equality of the spectra $\la_n = \tilde \la_n$ $(n \ge 1)$ and $\mu_n = \tilde \mu_n$ $(n \in \mathbb Z_0)$ yields $p = \tilde p$ in $L_1(0, 1)$ and $q = \tilde q$ in $W_1^{-1}(0, 1)$.
\end{thm}

Next, we prove the following existence theorem for solution of Inverse Problem~\ref{ip:1}.

\begin{thm} \label{thm:sc}
Let $\{ \la_n \}_{n \ge 1}$ and $\{ \mu_n \}_{n \in \mathbb Z_0}$ be any complex numbers satisfying the following conditions:
\begin{enumerate}
\item $\la_n \ne \la_k$ $(n \ne k)$ and $\mbox{Re} \, \la_n < 0$, $n \ge 1$.
\item The asymptotic relations hold: 
\begin{align} \label{asymptla}
\la_n & = -\bigl( \nu(n + 1/6)\bigr)^3 + 2 p_0 \nu (n + 1/6) + n \varkappa_n, \quad n \ge 1; \\ \label{asymptmu} 
\mu_n & = (\nu n)^3 - 2 p_0 \nu n + n \kappa_n, \quad n \in \mathbb Z_0, 
\end{align}
where $\nu := \dfrac{2 \pi}{\sqrt 3}$, $p_0 \in \mathbb R$, $\{ \varkappa_n \}_{n \ge 1} \in l_2$, and $\{ \ln(|n|) \kappa_n \}_{n \in \mathbb Z_0} \in l_2$.
\item For all $n \ge 1$, there holds
\begin{equation} \label{reld}
\delta^3(-\overline{\la_n}) = \overline{\Delta(\la_n)} \Delta(-\overline{\la_n}) 
\end{equation}
where 
\begin{equation} \label{prod}
\delta(\la) = \frac{1}{2} \prod_{n = 1}^{\infty} \frac{\la_n - \la}{\la_n^0}, \quad \Delta(\la) = \prod_{n \in \mathbb Z_0} \frac{\mu_n - \la}{\mu_n^0},
\end{equation}
and $\{ \la_n^0 \}_{n \ge 1}$, $\{ \mu_n^0 \}_{n \in \mathbb Z_0}$ are the spectra of the boundary value problems with the zero coefficients: $\mathcal L_0 := \mathcal L(0, 0)$, $\mathcal M_0 := \mathcal M(0,0)$.
\end{enumerate}
Then, there exist unique real-valued functions $p \in L_2(0, 1)$ and $q \in W_2^{-1}(0,1)$ such that $\{ \la_n \}_{n \ge 1}$ and $\{ \mu_n \}_{n \in \mathbb Z_0}$ are the eigenvalues of $\mathcal L(p, q)$ and $\mathcal M(p, q)$, respectively, and $p_0 = \int_0^1 p(x) \, dx$.

If refined asymptotic relations hold:
\begin{align} \label{asymptla+}
\la_n & = -\bigl( \nu(n + 1/6)\bigr)^3 + 2 p_0 \nu (n + 1/6) + \mathrm{i} q_0 - 2 p_1 + \varkappa_n, \quad n \ge 1; \\ \label{asymptmu+} 
\mu_n & = (\nu n)^3 - 2 p_0 \nu n + \mathrm{i} q_0 + \kappa_n, \quad n \in \mathbb Z_0, 
\end{align}
where $p_1$ and $q_0$ are real, $\{ \varkappa_n \}_{n \ge 1} \in l_2$ and $\{ \ln(|n|) \kappa_n \}_{n \in \mathbb Z_0} \in l_2$, then $p \in W_2^1[0,1]$, $q \in L_2(0,1)$, $p_1 = p(0) = p(1)$, and $q_0 = \int_0^1 q(x) \, dx$.
\end{thm}

We emphasize that Theorem~\ref{thm:sc} requires neither the reality of $\{ \la_n \}_{n \ge 1}$ and $\{ \mu_n \}_{n \in \mathbb Z_0}$ nor the simplicity of $\{ \mu_n \}_{n \in \mathbb Z_0}$. Moreover, we will show that the asymptotics \eqref{asymptla}--\eqref{asymptmu} with $\{ \varkappa_n \} \in l_2$ and $\{ \kappa_n \} \in l_2$, as well as the relation \eqref{reld} hold for the spectra $\{ \la_n \}_{n \ge 1}$ and $\{ \mu_n \}_{n \in \mathbb Z_0}$ of any real-valued functions $p \in L_2(0,1)$ and $q \in W_2^{-1}(0,1)$.

The proofs of the main theorems are based on the general inverse problem theory for higher-order differential operators \cite{Yur02}, including the ones with distribution coefficients \cite{Bond21, GYB25, Bond23}. We establish the relation between the given spectra $\{ \la_n \}_{n \ge 1}$, $\{ \mu_n \}_{n \in \mathbb Z_0}$ and the elements of the Weyl-Yurko matrix, which leads to the uniqueness theorem. In order to prove the existence theorem, we reduce Inverse Problem~\ref{ip:1} to the recovery of $p$ and $q$ from the spectral data $\{ \la_n, \be_n \}_{n \ge 1}$ introduced for the third-order problem $\mathcal L$ in \cite{Bond23}.
We find asymptotical and structural assumptions on $\{ \la_n \}_{n \ge 1}$ and $\{ \mu_n \}_{n \in \mathbb Z_0}$ that are sufficient for $\{ \la_n, \be_n \}_{n \ge 1}$ to satisfy the inverse problem solvability conditions from \cite{Bond23}.
Our technique is based on recovering entire characteristic functions from their zeros by Hadamard's Factorization Theorem and on the asymptotical analysis using the Birkhoff solutions, which were recently constructed for differential equations with distribution coefficients in \cite{SS20}.

The paper is organized as follows. In Section~\ref{sec:prelim}, we present the Mirzoev-Shkalikov regularization and basic results of the inverse spectral theory for equation \eqref{eqv}, as well as other preliminaries. In Section~\ref{sec:asympt}, the necessary asymptotics of the eigenvalues $\{ \mu_n \}_{n \in \mathbb Z_0}$ and of some characteristic functions are obtained. In Sections~\ref{sec:uniq} and~\ref{sec:exist}, we prove Theorems~\ref{thm:uniq} and~\ref{thm:sc}, respectively.

\smallskip

Throughout the paper, we use the notations:
\begin{itemize}
\item The prime $y'(x, \la)$ denotes the derivative with respect to $x$ and the dot $\dot y(x, \la)$, with respect to $\la$.
\item $\de_{k,j}$ is the Kronecker delta.
\item $f^{\star}(\la) = \overline{f(-\overline{\la})}$.
\item In estimates, the same symbol $C$ denotes various positive constants.
\end{itemize}

\section{Preliminaries} \label{sec:prelim}

Let us begin from the regularization of equation \eqref{eqv} in terms of quasi-derivatives (see \cite{MS16-talk, MS19, Bond23}). 

The claim $q \in W_s^{-1}(0, 1)$ ($s = 1, 2$) means that $q = \sigma'$ in the sense of distributions, where $\sigma \in L_s(0, 1)$. The $1$-periodicity of $p$ and $q$ means that
$p(x + 1) := p(x)$ and $\sigma(x + 1) := \sigma(x) + c$ a.e. on $(0, 1)$, where $c$ can be any fixed constant.

Introduce the matrix function $F(x) = [f_{k,j}(x)]_{k,j = 1}^3$ associated with equation \eqref{eqv}:
$$
F(x) = \begin{bmatrix}
0 & 1 & 0 \\
-(p + \mathrm{i} \sigma) & 0 & 1 \\
0 & (-p + \mathrm{i} \sigma) & 0 
\end{bmatrix},
$$
where $\sigma$ is a fixed real-valued antiderivative of $q$. Although the choice of $\sigma$ is non-unique, it does not influence the further arguments.

Introduce the quasi-derivatives in a standard way:
\begin{equation} \label{quasi}
y^{[0]} := y, \quad y^{[k]} := (y^{[k-1]})' - \sum_{j = 1}^k f_{k,j} y^{[j-1]}, \quad k = 1, 2, 3.
\end{equation}

Clearly, we have
$$
y^{[j]} = y^{(j)}, \quad j = 0, 1, \quad y^{[2]} = y'' + (p + \mathrm{i} \sigma) y, \quad
y^{[3]} = (y^{[2]})' + (p - \mathrm{i} \sigma) y'.
$$

Introduce the domain
$$
D_F := \bigl\{ y \colon y^{[k]} \in AC[0,2], \, k = 0, 1, 2 \bigr\}.
$$

According to \cite{MS19}, for any $y \in D_F$, the differential expression $y''' + (p y)' + p y' + \mathrm{i} q y$ produces a regular generalized function, which equals $y^{[3]}$. Thus, we call $y$ a solution of \eqref{eqv} if $y \in D_F$ and $y^{[3]} = \la y$ a.e. on $(0, 2)$.

Next, we present the results on inverse spectral problems for equation \eqref{eqv} basing on \cite{Bond21, Bond23}.
For $k = 1, 2, 3$, denote by $C_k(x, \la)$ and $\Phi_k(x, \la)$ the solutions of equation \eqref{eqv} satisfying the corresponding initial conditions
\begin{equation} \label{initC}
C_k^{[j-1]}(0, \la) = \de_{k,j}, \quad k,j = 1, 2, 3,
\end{equation}
and the boundary conditions
$$
\Phi_k^{[j-1]}(0, \la) = \de_{k,j}, \quad j = \overline{1,k}, \qquad
\Phi_k^{[s-1]}(1, \la) = 0, \quad s = \overline{1,3-k}.
$$

Following \cite{Bond21, Bond23}, introduce the Weyl-Yurko matrix $M(\la) := [M_{j,k}(\la)]_{j,k = 1}^3$,
$M_{j,k}(\la) := \Phi_k^{[j-1]}(0,\la)$. Obviously, the Weyl-Yurko matrix has lower-triangular structure:
$$
M(\la) = \begin{bmatrix}
1 & 0 & 0 \\
M_{2,1}(\la) & 1 & 0 \\
M_{3,1}(\la) & M_{3,2}(\la) & 1
\end{bmatrix}. 
$$
Due to \cite{Bond21, Bond23}, the lower-diagonal elements are represented in the form 
\begin{equation} \label{relMjk}
M_{j,k}(\la) = -\frac{\Delta_{j,k}(\la)}{\Delta_{k,k}(\la)}, \quad 1 \le k < j \le 3,
\end{equation}
where 
\begin{equation} \label{Deltakk}
\Delta_{k,k}(\la) := \det\bigl( [C_j^{[3-s]}(1,\la)]_{s,j = k+1}^3\bigr) 
\end{equation}
and $\Delta_{j,k}(\la)$ is obtained from $\Delta_{k,k}(\la)$ by replacing $C_j$ by $C_k$.

For each fixed $x \in [0, 2]$ and $k, j = 1, 2, 3$, the functions $C_k^{[j-1]}(x, \la)$ are entire in $\la$ of order not greater than $1/3$, and so do $\Delta_{j,k}(\la)$. Hence, the elements $M_{j,k}(\la)$ of the Weyl-Yurko matrix are meromorphic in $\la$, and their poles belong to the zero set of the corresponding function $\Delta_{k,k}(\la)$ ($k = 1, 2$). By construction of $\Delta_{k,k}(\la)$, its zeros coincide with the eigenvalues (counting with multiplicities) of the boundary value problem for equation \eqref{eqv} with the corresponding conditions
$$
y^{[j-1]}(0) = 0, \quad j = \overline{1,k}, \qquad
y^{[s-1]}(1) = 0, \quad s = \overline{1,3-k}.
$$
In particular, the zeros of $\Delta_{2,2}(\la) = C_3(1, \la)$ are $\{ \la_n \}_{n \ge 1}$

In order to prove Theorem~\ref{thm:uniq}, we will need the following uniqueness result from \cite{GYB25}, which generalizes \cite[Theorem~2.5.1]{Yur02} to differential operators with distribution coefficients.

\begin{prop}[\hspace*{-3pt}\cite{GYB25}, Theorem~4.1] \label{prop:Weyl}
Suppose that $\Delta_{1,1}(\la)$ and $\Delta_{2,2}(\la)$ have no common zeros, and so do $\tilde \Delta_{1,1}(\la)$ and $\tilde \Delta_{2,2}(\la)$. Then the assumption $M_{k+1,k}(\la) \equiv \tilde M_{k+1, k}(\la)$ for $k = 1, 2$ implies $p = \tilde p$ in $L_1(0,1)$ and $q = \tilde q$ in $W_1^{-1}(0,1)$.
In other words, under the hypothesis of this proposition, the elements $M_{2,1}(\la)$ and $M_{3,2}(\la)$ of the Weyl-Yurko matrix uniquely specify the coefficients $p$ and $q$ of equation \eqref{eqv}.
\end{prop}

Proposition~\ref{prop:Weyl} is valid for any complex-valued $p \in L_1(0, 1)$ and $q \in W_1^{-1}(0, 1)$. In our case, $p$ and $q$ are real-valued, so $M_{2,1}(\la) = M_{3,2}^{\star}(\la)$, and the zeros of $\Delta_{1,1}(\la)$ are $\{ -\overline{\la_n} \}_{n \ge 1}$. Thus, Proposition~\ref{prop:Weyl} implies the following corollary.

\begin{cor} \label{cor:M32}
If $\la_n \ne -\overline{\la_k}$ for all $n, k \ge 1$, then $M_{3,2}(\la)$ uniquely specifies $p \in L_1(0, 1)$ and $q \in W_1^{-1}(0,1)$.
\end{cor}

Suppose that, additionally, the eigenvalues $\{ \la_n \}_{n \ge 1}$ are simple. Following \cite{Bond23}, introduce the weight numbers:
\begin{equation} \label{defbe}
\be_n := -\Res_{\la = \la_n} M_{3,2}(\la) = \frac{C_2(1, \la_n)}{\dot C_3(1, \la_n)}, \quad n \ge 1.
\end{equation}

It is known from \cite{Bond23} that, the case $p \in L_2(0, 1)$ and $q \in W_2^{-1}(0,1)$, the eigenvalues $\{ \la_n \}_{n \ge 1}$ have the asymptotics \eqref{asymptla}, the weight numbers $\{ \be_n \}_{n \ge 1}$ are non-zeros and have the asymptotics
\begin{equation} \label{asymptbe}
\be_n = 3 \la_n \left( 1 + \frac{\eta_n}{n}\right), \quad \{ \eta_n \} \in l_2, \quad n \ge 1.
\end{equation}

We call $\{ \la_n, \be_n \}_{n \ge 1}$ the spectral data of the problem $\mathcal L$. Consider the inverse spectral problem:

\begin{ip} \label{ip:2}
Given the spectral data $\{ \la_n, \be_n \}_{n \ge 1}$, find $p$ and $q$.   
\end{ip}

Theorem~2.5 from \cite{Bond23} implies the following existence result for Inverse Problem~\ref{ip:2}.

\begin{prop}[\hspace*{-3pt}\cite{Bond23, Bond26}] \label{prop:sc}
Let $\{ \la_n, \be_n \}_{n \ge 1}$ be any complex numbers satisfying the following conditions:
\begin{enumerate}
\item $\la_n \ne \la_k$ $(n \ne k)$, $\mbox{Re} \, \la_n < 0$
and $\be_n \ne 0$ for all $n \ge 1$.
\item The asymptotic relations \eqref{asymptla} and \eqref{asymptbe} hold.
\end{enumerate}
Then, there exist unique real-valued functions $p \in L_2(0, 1)$ and $q \in W_2^{-1}(0, 1)$ such that $\{ \la_n, \be_n \}_{n \ge 1}$ are the spectral data of $\mathcal L(p, q)$, and $p_0 = \int_0^1 p(x) \, dx$.

If the values $\{ \la_n \}_{n \ge 1}$ satisfy the refined asymptotics \eqref{asymptla+} and
\begin{equation} \label{asymptbe+}
\be_n = 3\la_n \left(1 + \frac{p_0 + p_1}{2 \pi^2 n^2} + \frac{\eta_n}{n^2} \right), \quad \{ \eta_n \} \in l_2,
\end{equation}
then $p \in W_2^1[0,1]$, $q \in L_2(0, 1)$, $q_0 = \int_0^1 q(x) \, dx$, and $p_1 = p(0) = p(1)$.
\end{prop}

In addition, we formulate the following corollary of Phragmen-Lindel\"of's theorem and Liouville's theorem, needed in the proofs.

\begin{prop}[\hspace{-3pt} \cite{Yur16}, Corollary~5.1] \label{prop:PL}
If an entire function $f(\la)$ of order less then $1/2$ tends to zero as $|\la| \to \infty$ on a certain ray $\arg \la = c_0$, then $f(\la) \equiv 0$.
\end{prop}

\section{Asymptotics} \label{sec:asympt}

In this section, we obtain the necessary asymptotics of the eigenvalues $\{ \mu_n \}_{n \in \mathbb Z_0}$ and of characteristic functions. Our approach is based on the standard method \cite{Nai68} taking into account recent developments \cite{SS20, BK21, Bond22, KS24}. So, we outline the technique briefly.

Put $\la = \rho^3$. Consider the sector
\begin{equation*} 
\Gamma := \bigl\{ \rho \in \mathbb C \colon 0 < \arg \rho < \pi/3 \bigr\}
\end{equation*}
and the corresponding extended sector
$$
\Gamma_{h,\rho_*} :=  \left\{ \rho \in \mathbb C \colon \rho + h \exp(\mathrm{i}\pi/6)   \in \Gamma, \, |\rho| > \rho^*\right\}, \quad h > 0, \, \rho_* > 0.
$$

Put $\om_1 := \exp(2\pi\mathrm{i}/3)$, $\om_2 := \exp(-2\pi\mathrm{i}/3)$, $\om_3 := 1$. Obviously,
\begin{equation*} 
\mbox{Re} \, (\rho \om_1) < \mbox{Re} \, (\rho \om_2) < \mbox{Re} \, (\rho \om_3), \quad \rho \in \Gamma.
\end{equation*}

The notation $\eps(\rho)$ will be used for various functions defined on a given set $G$ such that $\eps(\rho) = o(1)$ as $|\rho| \to \infty$, $\rho \in G$, 
and, in the case $p \in L_2(0, 1)$ and $q \in W_2^{-1}(0,1)$, there holds $\{ \eps(\rho_n) \} \in l_2$ for any non-condensing sequence $\{ \rho_n \} \subset G$ and each fixed $x \in [0,2]$. A sequence $\{ \rho_n \}_{n = 1}^{\infty}$ is called non-condensing if
$$
\sup_{r > 0} (N(r + 1) - N(r)) < \infty, \quad N(r) := \# \{ n \in \mathbb N \colon |\rho_n| \le r \}
$$

Basing on \cite{SS20, Yur22}, we deduce the following lemma.

\begin{lem} \label{lem:Birk}
For every $h > 0$, there exist $\rho_* > 0$ and a fundamental system of solutions $\{ y_k(x, \rho) \}_{k = 1}^3$ of equation \eqref{eqv} with the following properties $(k = 1, 2, 3, \, j = 0, 1, 2)$:
\begin{enumerate}
\item The functions $y_k^{[j]}(x, \rho)$ are continuous by $x \in [0,2]$ and by $\rho \in \overline{\Gamma}_{h,\rho_*}$.
\item For each fixed $x \in [0,2]$, the functions $y_k^{[j]}(x, \rho)$ are analytic in $\Gamma_{h,\rho_*}$.
\item The asymptotic relation holds:
\begin{equation} \label{asympty}
y_k^{[j]}(x, \rho) = (\rho \om_k)^j \exp(\rho \om_k x) \left( 1 + \frac{\tau(x)}{\rho \om_k} + \frac{\eps_{k,j}(x, \rho)}{\rho}\right), 
\end{equation}
where $\tau(x) := -\frac{2}{3} \int_0^x p(t) \, dt$ and the functions $\eps_{k,j}(x, \rho)$ have the properties of $\eps(\rho)$ for $\rho \in \overline{\Gamma}_{h, \rho_*}$ and each fixed $x \in [0, 2]$.
\item If $p \in W_2^1$ and $q \in L_2$, the refined asymptotic relation is satisfied:
\begin{equation} \label{asympty+}
y_k^{[j]}(x, \rho) = (\rho \om_k)^j \exp(\rho \om_k x) \left( 1 + \frac{\tau(x)}{\rho \om_k} + \frac{\varsigma(x) + j \tau'(x)}{(\rho \om_k)^2} + \frac{\eps_{k,j}(x, \rho)}{\rho^2}\right),
\end{equation}
where $\eps_{k,j}(x, \rho)$ have the same properties as in p.~3 and
$$
\varsigma(x) := \frac{2}{9} \left( \int_0^x p(t) \, dt \right)^2 + \frac{1}{3} p(x) - \frac{\mathrm{i}}{3} \int_0^x q(t) \, dt.
$$
\end{enumerate}
\end{lem}

By virtue of Lemma~\ref{lem:Birk} and the initial conditions \eqref{initC}, for $x \in [0, 2]$, $\la = \rho^3$, and $\rho \in \overline{\Gamma}_{h,\rho_*}$, there holds
\begin{equation} \label{expCk}
C_k(x, \la) = \sum_{j = 1}^3 a_{j,k}(\rho) y_j(x, \rho), \quad k = 1, 2, 3,
\end{equation}
where 
\begin{equation} \label{defA}
A(\rho) = \bigl[a_{j,k}(\rho)\bigr]_{j,k = 1}^3 = 
\begin{bmatrix}
y_1(0, \rho) & y_2(0, \rho) & y_3(0, \rho) \\
y_1'(0, \rho) & y_2'(0, \rho) & y_3'(0, \rho) \\
y_1^{[2]}(0, \rho) & y_2^{[2]}(0, \rho) & y_3^{[2]}(0, \rho) 
\end{bmatrix}^{-1}.
\end{equation}

Substituting \eqref{asympty} into \eqref{defA}, we get
\begin{equation} \label{asympta}
a(\rho) := \det A(\rho) =  \frac{a_0}{\rho^3} \left( 1 + \frac{\eps(\rho)}{\rho}\right), \quad a_0 := \frac{\mathrm{i}}{3 \sqrt 3},
\quad \rho \in \overline{\Gamma}_{h,\rho_*}.
\end{equation}

Clearly, the eigenvalues $\{ \mu_n \}_{n \in \mathbb Z_0}$ of the problem $\mathcal M$ \eqref{eqv}, \eqref{bc2} (counting with multiplicities) coincide with the zeros of the characteristic function
\begin{equation} \label{Delta1}
\Delta(\la) = \begin{vmatrix}
C_1(0, \la) & C_2(0, \la) & C_3(0, \la) \\
C_1(1, \la) & C_2(1, \la) & C_3(1, \la) \\
C_1(2, \la) & C_2(2, \la) & C_3(2, \la)
\end{vmatrix}.
\end{equation}

Using \eqref{expCk}, we obtain
\begin{equation} \label{defD}
\Delta(\la) = D(\rho) a(\rho), \quad D(\rho) := \begin{vmatrix}
y_1(0, \rho) & y_2(0, \rho) & y_3(0, \rho) \\
y_1(1, \rho) & y_2(1, \rho) & y_3(1, \rho) \\
y_1(2, \rho) & y_2(2, \rho) & y_3(2, \rho)
\end{vmatrix}
\end{equation}
for $\la = \rho^3$, $\rho \in \overline{\Gamma}_{h,\rho_*}$.

\begin{lem} \label{lem:asymptmu}
For $p \in L_1$ and $q \in W_1^{-1}$, the eigenvalues $\{ \mu_n \}_{n \in \mathbb Z_0}$ possess the asymptotics \eqref{asymptmu} with $p_0 = \int_0^1 p(x) \, dx$ and $\kappa_n = o(1)$ as $|n| \to \infty$. 
\begin{itemize}
\item If $p \in L_2$ and $q \in W_2^{-1}$, then \eqref{asymptmu} holds  with $\{ \kappa_n \}_{n \in \mathbb Z_0} \in l_2$.
\item For $p \in W_2^1$ and $q \in L_2$, the asymptotic relation \eqref{asymptmu+} holds with $p_0 = \int_0^1 p(x) \, dx$, $q_0 = \int_0^1 q(x) \, dx$, $p_1 = p(0) = p(1)$, and
$\{ \kappa_n \}_{n \in \mathbb Z_0} \in l_2$.
\end{itemize}
\end{lem}

\begin{proof}
Similarly to the Counting Lemma in \cite{BK21}, one can check that, for every sufficiently large integer $N > 0$, the function $\Delta(\la)$ has exactly $2N$ zeros $\{ \mu_n \}_{n = \overline{-N, N} \setminus \{ 0 \}}$ (counting with multiplicities) in the circle $|\la| < \nu^3 (N + 1/4)^3$, $\nu = \dfrac{2 \pi}{\sqrt 3}$. For sufficiently large $|n|$, we have $\mu_n = \theta_n^3$, where $\theta_n$ is a zero of $D(\rho)$ \eqref{defD}. Indeed, the function $a(\rho)$ has no zeros as $|\rho| \to \infty$ due to~\eqref{asympta}.

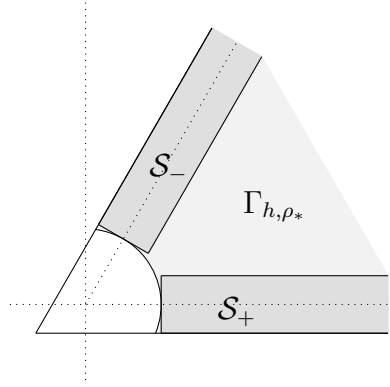
\begin{figure}[h!]
\centering
\begin{tikzpicture}
\filldraw[gray!10] (-0.6582, -0.38) -- (4, -0.38) -- (4, 0.38) -- (2.3291, 3.2741) -- (1.6709, 3.6541) -- cycle;
\filldraw[gray!25] (1, -0.38) -- (4, -0.38) -- (4, 0.38) --  (1, 0.38) -- cycle;
\filldraw[gray!25] (0.8291, 0.6760) -- (2.3291, 3.2741) -- (1.6709, 3.6541) -- (0.1709, 1.0560) -- cycle;
\filldraw[white] circle(1);

\draw[dotted] (-1, 0) edge (4, 0);
\draw[dotted] (0, -1) edge (0, 4);

\draw[dotted] (0, 0) edge (2.0, 3.4641);
\draw (-0.6582, -0.38) edge (1.6709, 3.6541);
\draw (-0.6582, -0.38) edge (4, -0.38);
\draw (2.5, 1.3) node{$\Gamma_{h,\rho_*}$};
\draw (0.925, -0.38) arc (-22.33:82.35:1cm);
\draw (1, -0.38) edge (4, -0.38);
\draw (1, 0.38) edge (4, 0.38);
\draw (1, -0.38) edge (1, 0.38);
\draw (2, -0.1) node{$\mathcal S_+$};
\draw (1.1, 1.8) node{$\mathcal S_-$};
\draw (0.8291, 0.6760) edge (2.3291, 3.2741);
\draw (0.1709, 1.0560) edge (1.6709, 3.6541);
\draw (0.8291, 0.6760) edge (0.1709, 1.0560);
\end{tikzpicture}
\caption{Strips $\mathcal S_+$ and $\mathcal S_-$}
\label{img:strips}
\end{figure}

Consider $\rho$ in the strip (see Fig.~\ref{img:strips}):
\begin{equation} \label{strip}
\mathcal S_+ := \bigl\{ \rho \in \mathbb C \colon \mbox{Re}\, \rho \ge h_1, \, |\mbox{Im} \rho| \le h_2 \bigr\}, \quad \mathcal S_+ \subset \overline{\Gamma}_{h, \rho_*}, 
\end{equation}
for suitable $h, \, \rho_* > 0$. Substituting \eqref{asympty} into \eqref{defD}, we obtain the asymptotics
$$
D(\rho) = y_3(2, \rho) \left( \exp(\rho \om_2) \biggl( 1 + \frac{\tau(1)}{\rho \om_2} \biggr)- \exp(\rho \om_1) \biggl( 1 + \frac{\tau(1)}{\rho \om_1}\biggr) + \frac{\eps(\rho)}{\rho}\right), \quad \rho \in \mathcal S_+.
$$
Applying the standard approach based on Rouche's Theorem (see \cite[Theorem~1.3]{FY01}), one can show that $D(\rho)$ has exactly one zero in a neighborhood of $\theta_n^0 = \nu n$ for each sufficiently large positive integer $n$. Specifically,
$$
\theta_n = \nu n + \frac{\tau(1)}{\nu n} + \frac{\eps_n}{n}, \quad \eps_n = o(1), \quad n \to \infty,
$$
and $\{ \eps_n \} \in l_2$ if $p \in L_2$ and $q \in W_2^{-1}$.
Analogous asymptotics can be obtained for negative values of $n$. Computing $\mu_n = \theta_n^3$ and taking the relation $\tau(1) = -\frac{2}{3} p_0$ into account, we arrive at \eqref{asymptmu}. In the case $p \in W_2^1$ and $q \in L_2$, we use the refined asymptotics \eqref{asympty+} to similarly derive~\eqref{asymptmu+}.
\end{proof}

Note that the asymptotic formulas \eqref{asymptmu} and \eqref{asymptmu+} are consistent with the ones in \cite{BK21, BK25} for the coefficients $p$ and $q$ in slightly different classes.

Using the standard method (see \cite[Theorem~1.4]{FY01}), one can easily show that the characteristic functions $\delta(\la) := C_3(1,\la)$ of $\mathcal L$ and $\Delta(\la)$ \eqref{Delta1} of $\mathcal M$ can be recovered from their zeros by formulas \eqref{prod}. In view of the asymptotics \eqref{asymptla} and \eqref{asymptmu}, the infinite products in \eqref{prod} converge absolutely and uniformly on compact sets.

\begin{lem} \label{lem:asymptDelta}
For $\la = \rho^3$, $\arg \rho = \pi/6$, $|\rho| \to \infty$, there hold
\begin{align*}
& \Delta_{1,1}(\la) \sim a_0 (\om_2 - \om_3) \rho^{-2} \exp(\rho(\om_2 + \om_3)), \quad
\Delta_{2,2}(\la) \sim a_0 (\om_2 - \om_3) \rho^{-2} \exp(\rho \om_3), \\
& \Delta_{3,2}(\la) \sim -a_0 (\om_2^2 - \om_1^2) \rho^{-1} \exp(\rho \om_3), \quad \Delta(\la) \sim a_0 \rho^{-3} \exp(\rho(\om_2 + 2 \om_3)),
\end{align*}
where the notation $f(\rho) \sim g(\rho)$ means that $\lim\limits_{|\rho| \to \infty} \dfrac{f(\rho)}{g(\rho)} = 1$.
\end{lem}

\begin{proof}
Using \eqref{Deltakk} and \eqref{expCk}, we obtain
\begin{align} \nonumber
\Delta_{1,1}(\la) & = \begin{vmatrix}
    C_1(0, \la) & C_2(0, \la) & C_3(0, \la) \\
    C_1'(1, \la) & C_2'(1, \la) & C_3'(1, \la) \\
    C_1(1, \la) & C_2(1, \la) & C_3(1, \la)
\end{vmatrix} = 
\begin{vmatrix}
    y_1(0, \rho) & y_2(0, \rho) & y_3(0, \rho) \\
    y_1'(1, \rho) & y_2'(1, \rho) & y_3'(1, \rho) \\
    y_1(1, \rho) & y_2(1, \rho) & y_3(1, \rho)
\end{vmatrix} a(\rho), \\ \label{expd}
\Delta_{2,2}(\la) & = \begin{vmatrix}
    C_1(0, \la) & C_2(0, \la) & C_3(0, \la) \\
    C_1'(0, \la) & C_2'(0, \la) & C_3'(0, \la) \\
    C_1(1, \la) & C_2(1, \la) & C_3(1, \la)
\end{vmatrix} = 
\begin{vmatrix}
    y_1(0, \rho) & y_2(0, \rho) & y_3(0, \rho) \\
    y_1'(0, \rho) & y_2'(0, \rho) & y_3'(0, \rho) \\
    y_1(1, \rho) & y_2(1, \rho) & y_3(1, \rho)
\end{vmatrix} a(\rho), \\ \nonumber
-\Delta_{3,2}(\la) & = \begin{vmatrix}
    C_1(0, \la) & C_2(0, \la) & C_3(0, \la) \\
    C_1^{[2]}(0, \la) & C_2^{[2]}(0, \la) & C_3^{[2]}(0, \la) \\
    C_1(1, \la) & C_2(1, \la) & C_3(1, \la)
\end{vmatrix} = 
\begin{vmatrix}
    y_1(0, \rho) & y_2(0, \rho) & y_3(0, \rho) \\
    y_1^{[2]}(0, \rho) & y_2^{[2]}(0, \rho) & y_3^{[2]}(0, \rho) \\
    y_1(1, \rho) & y_2(1, \rho) & y_3(1, \rho)
\end{vmatrix} a(\rho).
\end{align}
Substituting \eqref{asympty} and \eqref{asympta} into the above relations and into \eqref{defD}, we complete the proof.
\end{proof}

\section{Uniqueness} \label{sec:uniq}

This section focuses on the proof of Theorem~\ref{thm:uniq}. The main idea consists in constructing the element $M_{3,2}(\la)$ of the Weyl-Yurko matrix by using the given data $\{ \la_n \}_{n \ge 1}$ and $\{ \mu_n \}_{n \in \mathbb Z_0}$ of Inverse Problem~\ref{ip:1}.

For $k = 1, 2, 3$, denote by $S_k(x, \la)$ the solution of equation \eqref{eqv} under the initial conditions
\begin{equation} \label{initS}
S_k^{[j-1]}(1, \la) = \de_{k,j}, \quad k,j = 1, 2, 3.
\end{equation}

Let $y$ be a solution of \eqref{eqv}, and let $z$ be a solution of the equation
\begin{equation} \label{eqv*}
-z''' - (p z)' - p z' + \mathrm{i} q z = \lambda z, \quad x \in (0, 2).
\end{equation}

For $z$, define the quasi-derivatives similarly to \eqref{quasi} using the functions $f_{k,j}^{\star} = \overline{f_{k,j}}$ instead of $f_{k,j}$. Introduce the Lagrange bracket
$$
\langle z, y \rangle = z^{[2]} y - z' y' + z y^{[2]}.
$$
One can easily check that $\dfrac{d}{dx} \langle z, y \rangle = 0$, that is, the Lagrange bracket $\langle z, y \rangle$ does not depend on $x$.

Now, we are ready to prove the following lemma.

\begin{lem}
The eigenvalues $\{ \mu_n \}_{n \in \mathbb Z_0}$ of the boundary value problem $\mathcal M$ \eqref{eqv}, \eqref{bc2} coincide with the zeros of the characteristic function
\begin{equation} \label{Delta2}
\Delta(\la) = C_2(1, \la) C_3^{\star}(1, \la) + {C_3^{\star}}'(1, \la) C_3(1, \la).
\end{equation}
\end{lem}

\begin{proof}
Eigenfunctions of $\mathcal M$ have the form $y_n(x) = b_n S_2(x, \mu_n) + c_n S_3(x, \mu_n)$. Due to \eqref{initC}, \eqref{initS}, and the $1$-periodicity of $p$ and $q$, there holds $S_k(x, \la) = C_k(x - 1, \la)$ for $k = 2, 3$. Thus $y_n(x) = b_n C_2(x - 1, \mu_n) + c_n C_3(x - 1, \mu_n)$. Taking the boundary conditions $y(0) = y(2) = 0$ into account, we conclude that the eigenvalues of $\mathcal M$ coincide with the zeros of the determinant
\begin{equation} \label{DeltaCS}
\Delta(\la) = C_2(1, \la) S_3(0, \la) - S_2(0, \la) C_3(1, \la)
\end{equation}
of the linear system
$$
\begin{cases}
b C_2(1, \la) + c C_3(1, \la) = 0, \\
b S_2(0, \la) + c S_3(0, \la) = 0.
\end{cases}
$$

Next, note that $C_3^{\star}(x, \la)$ is a solution of \eqref{eqv*}. Consequently, the Lagrange brackets $\langle C_3^{\star}(x, \la), S_k(x, \la) \rangle$ ($k = 2, 3$) do not depend on $x$. Using \eqref{initC} and \eqref{initS}, we obtain
\begin{align*} 
& \langle C_3^{\star}, S_3 \rangle \big|_{x = 0} = S_3(0, \la), \quad \langle C_3^{\star}, S_3 \rangle \big|_{x = 1} = C_3^{\star}(1, \la), \\ \nonumber
& \langle C_3^{\star}, S_2 \rangle \big|_{x = 0} = S_2(0, \la), \quad \langle C_3^{\star}, S_2 \rangle \big|_{x = 1} = -{C_3^{\star}}'(1, \la).
\end{align*}
Hence 
\begin{equation} \label{CS3}
S_3(0, \la) = C_3^{\star}(1, \la), \quad S_2(0, \la) = -C_3^{\star}(1, \la), 
\end{equation}
which together with \eqref{DeltaCS} imply \eqref{Delta2}.
\end{proof}

Note that the functions \eqref{Delta1} and \eqref{Delta2} are equal to each other, so we use the same notation $\Delta(\la)$ for them. Indeed, it can be checked that both are equal to the same infinite product \eqref{prod}.

\begin{proof}[Proof of Theorem~\ref{thm:uniq}]
Since $\la_n = \tilde \la_n$ ($n \ge 1$) and $\mu_n = \tilde \mu_n$ ($n \in \mathbb Z_0$), the asymptotics \eqref{asymptmu} and the relations \eqref{prod} imply $p_0 = \tilde p_0$, $\delta(\la) = \tilde \delta(\la)$, and $\Delta(\la) = \tilde \Delta(\la)$. Recalling that $\delta(\la) = C_3(1, \la)$ and $\delta^{\star}(\la) = C_3^{\star}(1, \la)$, we derive from \eqref{Delta2}:
\begin{equation} \label{sm3}
C_2(1, \la) = \frac{\Delta(\la)}{\de^{\star}(\la)} - \frac{{C_3^{\star}}'(1,\la) \de(\la)}{\de^{\star}(\la)}.
\end{equation}

Note that the zeros of $\de^{\star}(\la)$ equal $\{ -\overline{\la_n}\}_{n \ge 1}$. Due to the assumption $\la_n \ne -\overline{\la_k}$ for all $n, k \ge 1$, we have $\de^{\star}(\la_n) \ne 0$. Consequently, it follows from \eqref{sm3} that the difference $(C_2 - \tilde C_2)(1, \la)$ has zeros $\{ \la_n \}_{n \ge 1}$ of (at least) the same multiplicities as $\de(\la)$. Therefore, the function $g(\la) := \dfrac{C_2(1,\la) - \tilde C_2(1, \la)}{\de(\la)}$ is entire in $\la$ of order not greater than $1/3$. 

Using the equalities $p_0 = \tilde p_0$, $C_2(1,\la) = \Delta_{3,2}(\la)$, $\de(\la) = \Delta_{2,2}(\la)$, and Lemma~\ref{lem:asymptDelta}, we obtain the asymptotics 
$$
C_2(1, \rho^3) - \tilde C_2(1, \rho^3) = o\bigl( \rho^{-2} \exp(\rho \om_3) \bigr),
$$
so $g(\rho^3) = o(1)$ as $|\rho| \to \infty$, $\arg \rho = \pi/6$. Proposition~\ref{prop:PL} implies that $g(\la) \equiv 0$, so $C_2(1, \la) = \tilde C_2(1, \la)$. Using \eqref{relMjk}, we get $M_{3,2}(\la) = \tilde M_{3,2}(\la)$. Applying Corollary~\ref{cor:M32} concludes the proof.
\end{proof}

In the case of simple eigenvalues $\{ \la_n \}_{n \ge 1}$, using \eqref{defbe} and \eqref{sm3}, we obtain
\begin{equation} \label{findbe}
\be_n = \frac{\Delta(\la_n)}{\dot \de(\la_n) \de^{\star}(\la_n)}, \quad n \ge 1.
\end{equation}

Thus, by using the relations \eqref{prod} and \eqref{findbe}, Inverse Problem~\ref{ip:1} is reduced to Inverse Problem~\ref{ip:2}.

Developing the technique of this section, we obtain the following auxiliary result.

\begin{lem} \label{lem:reld}
The functions $\de(\la) = C_3(1, \la)$ and $\Delta(\la)$ given by \eqref{Delta2} satisfy \eqref{reld}.
\end{lem}

\begin{proof}
Recall that the eigenvalues $\{ \la_n \}_{n \ge 1}$ of $\mathcal L$ are the zeros of $\de(\la) = C_3(1, \la)$. Therefore, using \eqref{Delta2}, we obtain
$$
\Delta(\la_n) = C_2(1, \la_n) C_3^{\star}(1, \la_n), \quad
\Delta(-\overline{\la_n}) = \overline{C_3'(1, \la_n)} C_3(1, -\overline{\la_n}),
$$
so
\begin{equation} \label{sm1}
\Delta(\la_n) \overline{\Delta(-\overline{\la_n})} = C_2(1, \la_n) C_3'(1, \la_n) \bigl( C_3^{\star}(1, \la_n)\bigr)^2.
\end{equation}

Consider the determinant
\begin{equation} \label{Theta1}
\Theta(\la) = \begin{vmatrix}
C_2(x, \la) & C_3(x, \la) & S_3(x, \la) \\
C_2'(x, \la) & C_3'(x, \la) & S_3'(x, \la) \\
C_2^{[2]}(x, \la) & C_3^{[2]}(x, \la) & S_3^{[2]}(x, \la)
\end{vmatrix} = -\Delta_{1,1}(\la),
\end{equation}
which does not depend on $x$. Using \eqref{initC}, \eqref{initS}, and \eqref{CS3}, we compute
\begin{align} \label{Theta2}
& x = 0 \colon \quad \Theta(\la) = S_3(0, \la) = C_3^{\star}(1, \la), \\ \nonumber
& x = 1 \colon \quad \Theta(\la_n) = \begin{vmatrix}
C_2(1, \la_n) & C_3(1, \la_n) \\ C_2'(1, \la_n) & C_3'(1, \la_n)
\end{vmatrix} = C_2(1, \la_n) C_3'(1, \la_n).
\end{align}
Hence 
\begin{equation} \label{sm2}
C_2(1, \la_n) C_3'(1, \la_n) = C_3^{\star}(1, \la_n) = \overline{\de(-\overline{\la_n})}.
\end{equation}
Substituting \eqref{sm2} into \eqref{sm1}, we obtain
$$
\overline{\de^3(-\overline{\la_n})} = \Delta(\la_n) \overline{\Delta(-\overline{\la_n})}.
$$
The complex conjugation implies \eqref{reld}.
\end{proof}

\section{Existence} \label{sec:exist}

In this section, we prove Theorem~\ref{thm:sc} by reducing Inverse Problem~\ref{ip:1} to Inverse Problem~\ref{ip:2} by the spectral data $\{ \la_n, \be_n \}_{n \ge 1}$. The most challenging parts of the proof are deriving the asymptotics \eqref{asymptbe} for the weight numbers from the asymptotics of $\la_n$ and $\mu_n$
and showing that $\{ \mu_n \}_{n \in \mathbb Z_0}$ are the eigenvalues of the corresponding problem $\mathcal M(p, q)$.

\begin{proof}[Proof of Theorem~\ref{thm:sc}]
Let $\{ \la_n \}_{n \ge 1}$ and $\{ \mu_n \}_{n \in \mathbb Z_0}$ be arbitrary complex numbers satisfying the conditions 1--3 of Theorem~\ref{thm:sc}. Construct the functions $\de(\la)$ and $\Delta(\la)$ by \eqref{prod}. In view of the asymptotics \eqref{asymptla} and \eqref{asymptmu}, the corresponding infinite products converge absolutely and uniformly on compact sets to analytic functions. Construct the numbers $\{ \be_n \}_{n \ge 1}$ by \eqref{findbe}. Our goal is to show that the data $\{ \la_n, \be_n \}_{n \ge 1}$ satisfy the conditions of Proposition~\ref{prop:sc}.

According to the hypothesis of Theorem~\ref{thm:sc}, the zeros of $\de(\la)$ are simple. Since $\mbox{Re} \, \la_n < 0$, then $\la_n \ne -\overline{\la_k}$ for all $n,k \ge 1$, so $\de^{\star}(\la_n) \ne 0$. Furthermore, it follows from \eqref{reld} that $\Delta(\la_n) \ne 0$. Hence $\be_n \ne 0$, $\be_n \ne \infty$, $n \ge 1$. It remains to prove the asymptotics \eqref{asymptbe}.

Put $\tilde p = p_0$, $\tilde q = 0$ and consider the characteristic function $\tilde \Delta(\la)$ of the problem $\tilde{\mathcal M} = \mathcal M(\tilde p, \tilde q)$. In the further discussion, the notation $k_*$ is used for various sufficiently large positive integers and $\{ \eps_n \}$, for various sequences of $l_2$.

\begin{lem} \label{lem:D1}
For $k \ge k_*$, there holds
$$
\Delta(\la_k) = \tilde \Delta(\la_k) \left( 1 + \frac{\eps_k}{k}\right), \quad \{ \eps_k \} \in l_2.
$$
\end{lem}

\begin{proof}
Using \eqref{prod} for $\Delta(\la)$ and $\tilde \Delta(\la)$, we obtain
$$
\Delta(\la) = \tilde \Delta(\la) \prod_{n \in \mathbb Z_0} \left( 1 + \frac{\mu_n - \tilde \mu_n}{\tilde \mu_n - \la}\right),
$$
so
\begin{equation} \label{smD}
\Delta(\la_k) = \tilde \Delta(\la_k) \prod_{n \in \mathbb Z_0} (1 + b_{n,k}), \quad b_{n,k} := \frac{\mu_n - \tilde \mu_n}{\tilde \mu_n - \la_k}.
\end{equation}
By Lemma~\ref{lem:asymptmu}, the sequence $\{ \tilde \mu_n \}_{n \in \mathbb Z_0}$ satisfies the asymptotics \eqref{asymptmu+}. Together with \eqref{asymptmu} for $\{ \mu_n \}_{n \in \mathbb Z_0}$ and \eqref{asympta}, this yields 
$$
|b_{n,k}| \le \frac{\eps_n}{k \ln(|n| + 1) \bigl(|n+k| + 1\bigr)}, \quad \eps_n > 0, \quad \{ \eps_n \} \in l_2,
$$
and $|b_{n,k}| < 1/2$ for $n \in\mathbb Z_0$ and sufficiently large $k \ge k_*$.

Denote 
$$
\xi_k := k \ln \prod_{n \in \mathbb Z_0} (1 + b_{n,k}).
$$
For $k \ge k_*$, we obtain
$$
|\xi_k| \le k \prod_{n \in \mathbb Z_0} |\ln(1 + b_{n,k})| \le C k \sum_{n \in \mathbb Z_0} |b_{n,k}| \le C \sum_{n \in \mathbb Z_0} \frac{\eps_n}{\ln(|n| + 1) \bigl( |n + k| + 1 \bigr)}.
$$

Schur's test implies that the operator $T$ given by $(T a)_k = \sum\limits_{n \in \mathbb Z_0} \dfrac{a_n}{\ln(|n| + 1)\bigl( |n+k| + 1\bigr)}$ is bounded from $l_2(\mathbb Z_0)$ to $l_2(\mathbb Z_0)$ (see \cite[Appendix C]{Bond26}). Hence $\{ \xi_k \}_{k \ge k_*} \in l_2$, so
$$
\prod_{n \in \mathbb Z_0} (1 + b_{n,k}) = \exp\left(\frac{\xi_k}{k} \right) =
1 + \frac{\eps_k}{k}, \quad \{ \eps_k \} \in l_2, \quad k \ge k_*.
$$
This together with \eqref{smD} yield the claim.
\end{proof}

Next, choose any complex numbers $\{ \tilde{\tilde{\be}}_n \}_{n \ge 1}$ such that the data $\{ \la_n, \tilde{\tilde \be}_n \}_{n \ge 1}$ satisfy the conditions~1 and~2 of Proposition~\ref{prop:sc}. For instance, one can put $\tilde{\tilde \be}_n := 3 \la_n$ if $\la_n \ne 0$ and $\tilde{\tilde \be}_n = 1$ otherwise. Then $\{ \la_n, \tilde{\tilde \be}_n \}_{n \ge 1}$ are the spectral data of some problem $\tilde{\tilde{\mathcal L}} = \mathcal L(\tilde{\tilde p}, \tilde{\tilde q})$ with $\tilde{\tilde p} \in L_2(0, 1)$, $\tilde{\tilde q} \in W_2^{-1}(0,1)$. 
Expand $\tilde{\tilde p}$ and $\tilde{\tilde q}$ periodically to $(0, 2)$ and consider the characteristic function $\tilde{\tilde \Delta}(\la)$ of the corresponding problem $\tilde{\tilde{\mathcal M}} = \mathcal M(\tilde{\tilde p}, \tilde{\tilde q})$.

\begin{lem} \label{lem:D2}
For $k \ge k_*$, there holds
$$
\tilde \Delta(\la_k) = \tilde{\tilde \Delta}(\la_k) \left( 1 + \frac{\eps_k}{k} \right), \quad \{ \eps_k \} \in l_2.
$$
\end{lem}

\begin{proof}
According to the asymptotics \eqref{asymptla}, we have $\la_k = \rho_k^3$, where $\rho_k \sim \nu k \exp(\pi \mathrm{i}/3)$ as $k \to \infty$. Consequently, for sufficiently large $k$, the values $\rho_k$ lie in the strip
$$
\mathcal S_- := \bigl\{ \rho \in \mathbb C \colon \rho \exp(-\pi\mathrm{i}/3) \in \mathcal S_+ \bigr\},
$$
where $\mathcal S_+$ is defined in \eqref{strip} (see Fig.~\ref{img:strips}). 
Therefore, using \eqref{asympty}, \eqref{asympta}, and the representation \eqref{defD} for $\tilde \Delta(\la_k)$, we obtain
\begin{align} \nonumber
\tilde \Delta(\la_k) & = \tilde a(\rho_k) \tilde D(\rho_k) = a_0 \rho_k^{-3} 
\begin{vmatrix}
\tilde y_2(1, \rho_k) & \tilde y_3(1, \rho_k) \\
\tilde y_2(2, \rho_k) & \tilde y_3(2, \rho_k) 
\end{vmatrix} 
\left( 1 + \frac{\eps_k}{k} \right) \\ \label{asymptDt}
& = a_0 \rho_k^{-3} \exp(\rho_k (\om_2 + 2 \om_3)) \left(\tilde g_1(\rho_k) - \exp(\rho_k (\om_2 - \om_3)) \tilde g_2(\rho_k) + \frac{\eps_k}{k} \right), \quad
k \ge k_*,
\end{align}
where 
$$
\tilde g_1(\rho) = 1 + \frac{\tilde \tau(1)}{\rho \om_2} + \frac{\tilde \tau(2)}{\rho \om_3}, \quad
\tilde g_2(\rho) = 1 + \frac{\tilde \tau(2)}{\rho \om_2} + \frac{\tilde \tau(1)}{\rho \om_3},
\quad 2 \tilde \tau(1) = \tilde \tau(2) = -\frac{4}{3} p_0.
$$

Recall that $\{ \la_k \}_{k \ge 1}$ are zeros of $\de(\la)$, which is the characteristic function of $\tilde{\tilde{\mathcal L}}$, so it satisfies \eqref{expd}. Hence, for sufficiently large $k$, we have $d(\rho_k) = 0$, where
$$
d(\rho) = \begin{vmatrix}
    y_1(0, \rho) & y_2(0, \rho) & y_3(0, \rho) \\
    y_1'(0, \rho) & y_2'(0, \rho) & y_3'(0, \rho) \\
    y_1(1, \rho) & y_2(1, \rho) & y_3(1, \rho)
\end{vmatrix}
$$
Using \eqref{asympty}, we derive
\begin{equation} \label{eqd}
d(\rho_k) = \rho_k \exp(\rho_k \om_3) \left( r_1(\rho_k) - \exp(\rho_k (\om_2 - \om_3)) r_2(\rho_k) + \frac{\eps_k}{k} \right) = 0, \quad k \ge k_*,
\end{equation}
where 
$$
r_1(\rho) = (\om_2 - \om_1) \left( 1 + \frac{\tau(1)}{\rho \om_3} \right), \quad r_2(\rho) = (\om_3 - \om_1) \left( 1 + \frac{\tau(1)}{\rho \om_3} \right), \quad \tau(1) = -\frac{2}{3} p_0.
$$
Solving equation \eqref{eqd}, we get
\begin{equation} \label{expk}
\exp(\rho_k(\om_2 - \om_3)) = \frac{r_1(\rho_k)}{r_2(\rho_k)} + \frac{\eps_k}{k}, \quad k \ge k_*.
\end{equation}
Substituting \eqref{expk} into \eqref{asymptDt}, we obtain
$$
\tilde \Delta(\la_k) = a_0 \rho_k^{-3} \exp(\rho_k(\om_2 + 2 \om_3)) \left( \tilde g_1(\rho_k) - \tilde g_2(\rho_k) \frac{r_1(\rho_k)}{r_2(\rho_k)} + \frac{\eps_k}{k}\right), \quad k \ge k_*.
$$
The similar asymptotics is valid for $\tilde{\tilde \Delta}(\la_k)$, which yields the claim of the lemma.
\end{proof}

\begin{lem}
The numbers $\{ \be_n \}_{n \ge 1}$ constructed by \eqref{findbe} have the asymptotics \eqref{asymptbe}.
\end{lem}

\begin{proof}
Since $\tilde{\tilde{\de}}(\la) \equiv \de(\la)$, then the relation \eqref{findbe} for the weight numbers $\{ \tilde{\tilde \be}_n \}_{n \ge 1}$ of the problem $\tilde{\tilde{\mathcal L}}$ takes the form
$$
\tilde{\tilde \be}_n = \frac{\tilde{\tilde \Delta}(\la_n)}{\dot \de(\la_n) \de^{\star}(\la_n)}, \quad n \ge 1,
$$
where $\tilde{\tilde \be}_n \ne 0$ and $\tilde{\tilde \Delta}(\la_n) \ne 0$.
Hence $\dfrac{\be_n}{\tilde{\tilde \be}_n} = \dfrac{\Delta(\la_n)}{\tilde{\tilde \Delta}(\la_n)}$, $n \ge 1$.
Combining Lemmas~\ref{lem:D1} and~\ref{lem:D2} implies 
$$
\frac{\Delta(\la_n)}{\tilde{\tilde \Delta}(\la_n)} = 1 + \frac{\eps_n}{n}, \quad n \ge 1.
$$
Thus $\be_n = \tilde{\tilde \be}_n \left( 1 + \dfrac{\eps_n}{n}\right)$, which yields \eqref{asymptbe}.
\end{proof}

Continue the proof of Theorem~\ref{thm:sc}.
Applying Proposition~\ref{prop:sc}, we conclude that there exist unique real-valued functions $p \in L_2(0,1)$ and $q \in W_2^{-1}(0,1)$ such that $\{ \la_n, \be_n \}_{n \ge 1}$ are the spectral data of $\mathcal L(p, q)$, and $\int_0^1 p(x) \, dx = p_0$. Expand $p(x)$ and $q(x)$ periodically to $(0, 2)$. 
It remains to show that the initially given numbers $\{ \mu_n \}_{n \in \mathbb Z_0}$ coincide with the eigenvalues of the problem $\mathcal M = \mathcal M(p, q)$. 

Recall that $\Delta(\la)$ is constructed by \eqref{prod} using $\{ \mu_n \}_{n \in \mathbb Z_0}$, and denote by $\Delta^{\diamond}(\la)$ the characteristic function of $\mathcal M(p, q)$ defined by \eqref{Delta2}.

\begin{lem} \label{lem:asymptDelta2}
There holds $\Delta(\la) - \Delta^{\diamond}(\la) = \Delta^{\diamond}(\la) o\bigl( \rho^{-1} \bigr)$ as $|\rho| \to \infty$, $\arg \rho = \pi/6$, $\la = \rho^3$.
\end{lem}

\begin{proof}
Using the representation \eqref{prod} for $\Delta(\la)$ and for $\Delta^{\diamond}(\la)$:
$$
\Delta^{\diamond}(\la) = \prod_{n \in \mathbb Z_0} \frac{\mu_n^{\diamond} - \la}{\mu_n^0},
$$
we obtain
$$
\Delta(\la) = \Delta^{\diamond}(\la) \Pi(\la), \quad
\Pi(\la) := \prod_{n \in \mathbb Z_0} (1 + p_n(\la)), \quad
p_n(\la) := \frac{\mu_n - \mu_n^{\diamond}}{\mu_n^{\diamond} - \la},
$$
so
\begin{equation} \label{difD}
\Delta(\la) - \Delta^{\diamond}(\la) = \Delta^{\diamond}(\la) \bigl( \Pi(\la) - 1 \bigr).
\end{equation}

Recall that $\{ \mu_n \}_{n \in \mathbb Z_0}$ satisfy \eqref{asymptmu} with $\{ \ln(|n|)\kappa_n \} \in l_2$ and $\{ \mu_n^{\diamond} \}_{n \in \mathbb Z_0}$ satisfy the similar asymptotics with $\{ \kappa_n^{\diamond} \} \in l_2$ according to Lemma~\ref{lem:asymptmu}. Hence $\{ |n|^{-1} (\mu_n - \mu_n^{\diamond}) \}_{n \in \mathbb Z_0} \in l_2$. Consequently, for $\la = \rho^3$, $\arg \rho = \pi/6$, $|\rho| \ge \rho_*$, we obtain
$$
|p_n(\la)| \le \frac{C \eps_n}{|\rho| \bigl( |n| + |\rho| \bigr)}, \quad \eps_n > 0, \quad \{ \eps_n \} \in l_2, \quad
n \in \mathbb Z_0.
$$
One can easily estimate
\begin{align*}
& |\ln \Pi(\la)| \le \sum_{n \in \mathbb Z_0} \ln |1 + p_n(\la)| \le C \sum_{n \in \mathbb Z_0} |p_n(\la)|, \\
& |\rho \ln \Pi(\la)| \le C \sum_{n \in \mathbb Z_0} \frac{\eps_n}{|n| + |\rho|} \le C |\rho|^{-1/2}.
\end{align*}
Thus, under our assumptions, we have $\ln\Pi(\la) = o\bigl(\rho^{-1}\bigr)$, which together with \eqref{difD} prove the lemma.
\end{proof}

\begin{lem} \label{lem:eqmu}
There holds $\mu_n = \mu_n^{\diamond}$ for all $n \ge 1$.
\end{lem}

\begin{proof}
The relation 
$$
\be_n = \frac{\Delta^{\diamond}(\la_n)}{\dot \de(\la_n) \de^{\star}(\la_n)}, \quad n \ge 1,
$$
together with \eqref{findbe} implies $\Delta(\la_n) = \Delta^{\diamond}(\la_n)$, $n \ge 1$. By Lemma~\ref{lem:reld}, we have
$$
\de^3(-\overline{\la_n}) = \overline{\Delta^{\diamond}(\la_n)} \Delta^{\diamond}(-\overline{\la_n}), \quad n \ge 1.
$$
Comparing this relation with \eqref{reld} (i.e. the condition~3 of Theorem~\ref{thm:sc}), we conclude that $\Delta(-\overline{\la_n}) = \Delta^{\diamond}(-\overline{\la_n})$, $n \ge 1$. Consequently, the following function is entire in $\la$ of order not greater than $1/3$:
$$
f(\la) = \frac{\Delta(\la) - \Delta^{\diamond}(\la)}{\de(\la) \de^{\star}(\la)}.
$$

In view of \eqref{Theta1} and \eqref{Theta2}, we have $\de^{\star}(\la) = C_3^{\star}(1, \la) = -\Delta_{1,1}(\la)$. So, applying Lemmas~\ref{lem:asymptDelta} and \ref{lem:asymptDelta2}, we obtain 
$$
f(\la) = \frac{\Delta^{\diamond}(\la) o\bigl( \rho^{-1} \bigr)}{\Delta_{1,1}(\la) \Delta_{2,2}(\la)} = o(1), \quad
|\rho| \to \infty, \quad \arg \rho = \frac{\pi}{6}, \quad \la = \rho^3.
$$
By Proposition~\ref{prop:PL}, we conclude that $f(\la) \equiv 0$, so $\Delta(\la) \equiv \Delta^{\diamond}(\la)$ and their zero sets coincide: $\mu_n = \tilde \mu_n^{\diamond}$, $n \ge 1$.
\end{proof}

\begin{remark}
The proof of Lemma~\ref{lem:eqmu} shows that the function $\Delta(\la)$ is uniquely specified by its values $\Delta(\la_n)$ and $\Delta(-\overline{\la_n})$, $n \ge 1$. Note that the values $\Delta(\la_n)$  ($n \ge 1$) are insufficient for the unique interpolation of $\Delta(\la)$. So, the condition \eqref{reld} in Theorem~\ref{thm:sc} is essential and cannot be omitted.
\end{remark}

Thus, the assertion of Theorem~\ref{thm:sc} for the case $p \in L_2$ and $q \in W_2^{-1}$ is proved. If the refined asymptotics \eqref{asymptla+} and \eqref{asymptmu+} are fulfilled, then one can similarly deduce the refined asymptotics \eqref{asymptbe+} for the numbers $\{ \be_n \}_{n \ge 1}$ constructed by \eqref{findbe}. Consequently, Proposition~\ref{prop:sc} yields $p \in W_2^1$, $q \in L_2$ and concludes the proof of Theorem~\ref{thm:sc}.
\end{proof}

\medskip

{\bf Funding.} This work was supported by Grant 24-71-10003 of the Russian Science Foundation, https://rscf.ru/en/project/24-71-10003/.

\end{document}